\documentclass{amsart}
\usepackage[utf8]{inputenc}
\usepackage{amsmath, amssymb, amsthm, amsfonts, mathtools, array, xfrac, float, multirow, indentfirst, url, enumerate, enumitem, comment, afterpage, pifont, makecell, xcolor, hyperref, bbm}
\usepackage[margin=1.5in]{geometry}

\usepackage[backend=bibtex, bibstyle=alphabetic, citestyle=alphabetic, sorting=nyt,maxbibnames=99, giveninits=false, isbn=false, url=false, doi=false, eprint=false]{biblatex}
\usepackage{thmtools}
\usepackage{cleveref}

\allowdisplaybreaks

\newtheorem{theorem}{Theorem}
\newtheorem{lemma}[theorem]{Lemma}
\newtheorem{proposition}[theorem]{Proposition}

\newtheorem{remark}{Remark}

\renewcommand{\Re}{\mathrm{Re}}
\renewcommand{\Im}{\mathrm{Im}}

\theoremstyle{definition}

\newcommand{\mr}{\mathrm}
\newcommand{\md}{\,\mathrm{d}}

\newcommand{\wh}{\widehat}
\newcommand{\wt}{\widetilde}

\newcommand{\eps}{\varepsilon}

\newcommand{\bb}{\mathbb}
\newcommand{\meas}{\mathrm{meas}}
\newcommand{\1}{\mathbbm{1}}

\makeatletter
\declaretheoremstyle
    [headformat={\NOTE}, 
    notebraces={}{}, 
    notefont=\bfseries, 
    preheadhook=\def\thmt@space{}, 
    numbered=no
    ]{namedtheorem}
\makeatother

\begin{document}

\author{Tianyu Zhao}

\title[On $r$-gaps between zeta zeros]{Density results for $r$-gaps between zeros of the Riemann zeta-function}

\address{
    Department of Mathematics, The Ohio State University, 231 West 18th
    Ave, Columbus, OH 43210, USA.
}
\email{zhao.3709@osu.edu}


\begin{abstract}
    Let $0<\gamma_1\leq \gamma_2\leq \ldots$ denote the positive ordinates of the non-trivial zeros of the Riemann zeta-function. A result first announced by Selberg states that there exist absolute constants $\Theta, \vartheta>0$ such that for each $r\in \bb{N}$,
    \[
    \limsup_{n\to \infty}\frac{\gamma_{n+r}-\gamma_n}{2\pi r/\log \gamma_n}\geq 1+\frac{\Theta}{r^\alpha} \qquad \text{and}\qquad \liminf_{n\to \infty}\frac{\gamma_{n+r}-\gamma_n}{2\pi r/\log \gamma_n}\leq 1-\frac{\vartheta}{r^\alpha}
    \]
    where $\alpha$ may be taken as $2/3$, or as $1/2$ if one assumes the Riemann hypothesis. This was recently proved by Conrey and Turnage-Butterbaugh under RH and by Inoue unconditionally. We prove that in fact a positive proportion of $r$-gaps are large (and small) to the above extent and provide explicit estimates for the sizes and proportions of these gaps. In the case $r=1$, this quantitatively improves an unconditional result of Simoni\v{c}--Trudgian--Turnage-Butterbaugh.
\end{abstract}

\maketitle

\section{Introduction}
Let $\zeta(s)$ denote the Riemann zeta-function. For $t>0$, let $N(t)$ be the number of non-trivial zeros $\rho=\beta+i\gamma$ of $\zeta(s)$ with $0<\gamma\leq t$ where any zero with $\gamma=t$ is counted with weight $1/2$. The Riemann--von Mangoldt formula says that
\begin{equation}\label{riemann-von mangoldt}
    N(t)=\frac{t}{2\pi}\log\frac{t}{2\pi e}+S(t)+\frac{7}{8}+O(t^{-1}), \quad t>1.
\end{equation}
Here if $t$ does not coincide with the ordinate of any zero,
\[
S(t):=\frac{1}{\pi}\mr{arg}\:\zeta\left(\frac{1}{2}+it\right)
\]
is the argument of $\zeta(s)$ obtained by continuous variation along the line segments joining $s=2$, $s=2+it$ and $s=1/2+it$, starting with $\mr{arg}\:\zeta(2)=0$. Otherwise we set
\[
S(t):=\lim_{\eps\to 0}\frac{S(t+\eps)+S(t-\eps)}{2}.
\]
Classical arguments show that $S(t)=O(\log t)$, so that
\begin{equation}\label{N(t)}
    N(t)\sim \frac{t}{2\pi}\log t, \quad t\to \infty.
\end{equation}
If we order the positive ordinates of the non-trivial zeros by $0<\gamma_1\leq \gamma_2\leq \ldots$, then \eqref{N(t)} implies that the average gap size between $\gamma_n$ and $\gamma_{n+1}$ is $2\pi/\log \gamma_n$ as $\gamma_n\to \infty$. The existence of small and large gaps is a subject of particular interest and has been extensively studied in the literature. 

In this paper, we work with the more general notion of $r$-gaps. In preparation, let us introduce some notations. Define for $r\in \bb{N}$
\[
\lambda_r:=\limsup_{n\to \infty}\frac{\gamma_{n+r}-\gamma_n}{2\pi r/\log \gamma_n}, \qquad \mu_r:=\liminf_{n\to \infty}\frac{\gamma_{n+r}-\gamma_n}{2\pi r/\log \gamma_n}
\]
and
\[
\wt{\lambda}_r:=\sup \left\{\lambda: D^+(\lambda,r)<1\right\}, \qquad \wt{\mu}_r:=\inf \left\{\mu: D^-(\mu,r)>0 \right\}
\]
where 
\[
D^+(\lambda,r):=\limsup_{T\to \infty} D(\lambda,r,T),\qquad D^-(\mu,r):=\liminf_{T\to \infty} D(\mu,r,T),
\]
\begin{equation}\label{def D}
    D(\mu,r,T):=\frac{\#\{\gamma_n\in [T,2T]: \gamma_{n+r}-\gamma_n\leq 2\pi \mu r/\log T\}}{\frac{T}{2
    \pi}\log T}.
\end{equation}
Note that $\lambda_r$ and $\mu_r$ measure large and small $r$-gaps that occur infinitely often, while $\wt{\lambda}_r$ and $\wt{\mu}_r$ measure those occurring a positive proportion of time. Plainly $\mu_r\leq \wt{\mu}_r\leq 1\leq \wt{\lambda}_r\leq \lambda_r$. In the classical setting $r=1$, Selberg \cite{Sel47} first observed, without publishing a proof, that $\wt{\mu}_1<1<\wt{\lambda}_1$. Heath-Brown provided a proof in \cite[Section 9.26]{Tit86}, which was recently made explicit by Simoni\v{c}, Trudgian and Turnage-Butterbaugh \cite[Theorem 1]{STT22}. In particular, they showed that 
\begin{equation}\label{sim-trud-turn constants}
    \begin{split}
        \wt{\lambda}_1>&\lambda:=1+10^{-10^{13}},\qquad D^+(\lambda,1)\leq 1-10^{-2\cdot 10^{13}}\\
        \wt{\mu}_1<&\mu:=1-10^{-3\cdot 10^{13}},\qquad D^-(\mu,1)\geq \frac{1}{2}\cdot 10^{-3\cdot 10^{13}}.
    \end{split}
\end{equation}
Much more progress has been made under the Riemann hypothesis (RH), which is assumed for the results mentioned below. Building on the series of works including \cite{MO84,CGG84,BMN10,Hall99,Hall02,Hall05,Pre16}, Inoue \cite{Inoue26} proved $\mu_1<0.50895$, while Bui and Milinovich \cite{BM18} proved $\lambda_1>3.18$. For results that hold for a positive proportion of zeros, Wu \cite{Wu14} managed to show $\wt{\lambda}_1>1.6989$ and $\wt{\mu}_1<0.6553$, refining results by Conrey--Ghosh--Goldston--Gonek--Heath-Brown \cite{CGGGH85} and Soundararajan \cite{Sound96}. Recently, Bui, Goldston, Milinovich and Montgomery \cite{BGMM23} obtained $\wt{\mu}_1<0.6039$ with a different method. 

In very recent preprints, $\mu_1\leq 1/2$ was proved by Bondarenko and Heap \cite{BH26} under RH, and by Durkan and Pearce-Crump \cite{DPC26} unconditionally. The latter authors also showed under RH that $\wt{\mu}_1\leq 1/2$ and $\lambda_1>3.63846723$.

For a general positive integer $r$, Selberg \cite[pp. 355]{Sel89} announced (again without proof) that there exist absolute constants $\Theta,\vartheta>0$ for which
\begin{equation}\label{lambda_r and mu_r}
    \lambda_r\geq 1+\frac{\Theta}{r^\alpha} \qquad \text{and} \qquad \mu_r\leq 1-\frac{\vartheta}{r^\alpha}
\end{equation}
where $\alpha$ may be taken as $1/2$ under RH and as $2/3$ unconditionally. See \cite{CTB18} for some history behind the statement of this result. Conrey and Turnage-Butterbaugh \cite{CTB18} supplied a conditional proof of \eqref{lambda_r and mu_r} based on the methods of Montgomery--Odlyzko \cite{MO84} and Conrey--Ghosh--Gonek \cite{CGG84}. Later, Inoue \cite{Inoue24} gave a complete proof including the unconditional part by using an $\Omega$-result of Tsang \cite{Tsa86} on the variation of $S(t)$ in short intervals, namely, $S(t+h)-S(t)$ where $h\ll \frac{1}{\log t}$. However, the sizes of the constants $\Theta,\vartheta$ have only been estimated conditionally. Conrey and Turnage-Butterbaugh \cite{CTB18} obtained $\Theta=0.5742$ and $\vartheta=0.2998$ and indicated some possible improvements (see their Remark 1). Inoue \cite{Inoue24} improved this to $\Theta=0.9064$ and $\vartheta=0.4846$ by optimizing certain calculations in \cite{CTB18}. Inoue, Koboyashi and Toma \cite{IKT25} showed that $\Theta=\vartheta=\sqrt{2}+o(1)$ is admissible as $r\to \infty$ by adapting Soundarajan's resonance method, originally devised to detect large values of $|\zeta(1/2+it)|$, for $S(t+h)-S(t)$.

The goal of the present paper is to extend \eqref{lambda_r and mu_r} to $\wt{\lambda}_r$ and $\wt{\mu}_r$ for all $r\geq 1$ and to work out the constants both with and without the assumption of RH. This is achieved by further exploiting Tsang's work \cite{Tsa86} and making explicit his $\Omega$-result.  

\begin{theorem}\label{theorem 1}
    There exist absolute constants $\Theta,\vartheta>0$ such that for all $r\geq 1$,
    \begin{equation}\label{theorem 1 inequality}
        \wt{\lambda}_r\geq 1+\frac{\Theta}{r^\alpha} \qquad \text{and} \qquad \wt{\mu}_r\leq 1-\frac{\vartheta}{r^\alpha}
    \end{equation}
    where one may take $\alpha=2/3$ unconditionally and $\alpha=1/2$ assuming RH. 
\end{theorem}

More precisely, we prove the following:

\begin{theorem}\label{theorem 2}
    Unconditionally, inequality \eqref{theorem 1 inequality} holds uniformly for all $r\geq 1$ with
    $\Theta=\vartheta=2.36\cdot10^{-6}$. When $r$ is sufficiently large, we may take $\Theta=\vartheta=0.01834$. 
    Under RH, $\Theta=0.2160$ and $\vartheta=0.1638$ are admissible for all $r\geq 1$, which can both be improved to $0.2643$ for large $r$.
\end{theorem}

To complement Theorem~\ref{theorem 2}, we also provide in \S\ref{section: completing the proof} the following examples of explicit proportion results. Unconditionally, we have
\[
D^-(1-10^{-6},1)>4.5\cdot 10^{-30}, \qquad D^+(1+10^{-6},1)<1-7.5\cdot 10^{-104},
\]
and assuming RH,
\[
D^-(1-0.15,1)>1.9\cdot 10^{-24}, \qquad D^+(1+0.15,1)<1-2.6\cdot 10^{-43}.
\]
It is possible to sharpen these bounds with more work; see Remark~\ref{possible improvement}. Already, for $r=1$, note that we have considerably improved the numerical constants \eqref{sim-trud-turn constants} obtained by Simoni\v{c}--Trudgian--Turnage-Butterbaugh \cite{STT22}. Although the underlying methods are different, an essential component of their work also comes from \cite{Tsa86}, namely, bounds for moments of $S(t+h)-S(t)$. Some of the explicit results established in \cite{STT22} will come in handy for us as well.

Theorem~\ref{theorem 1} is a consequence of the following result on extreme values of $S(t+h)-S(t)$:

\begin{theorem}\label{theorem measure S(t+h)-S(t)}
    There exist absolute constants $C,D,\kappa>0$ such that for all sufficiently large $T$, $h\in [\frac{1}{\log T},\frac{\kappa}{\log\log T}]$ and $0<V\leq C(h\log T)^{1-\alpha}$,
    \[
    \meas\left\{t\in [T,2T]: S(t+h)-S(t)\geq V\right\}\gg \frac{Te^{-DV^2}}{\log(h\log T+2)}
    \]
    where one may take $\alpha=2/3$ unconditionally and $\alpha=1/2$ assuming RH. The same statement holds for $-(S(t+h)-S(t))$.
\end{theorem}

This extends \cite[Theorem 3]{Tsa86}, which shows the existence of large values of size $V\asymp (h\log T)^{1-\alpha}$ without addressing the measure aspect. In the regime $h\asymp \frac{1}{\log T}$, observe from Theorem~\ref{theorem measure S(t+h)-S(t)} that $\pm(S(t+h)-S(t))$ is bounded from below by some positive absolute constant for a positive proportion of $t\in [T,2T]$ as $T\to \infty$. This will be crucial to establishing Theorem~\ref{theorem 1}.

Lastly, we point out that Dobner \cite{Dob24} proved an unconditional measure-theoretic result for large deviations of $S(t)$ itself, resembling the statement of Theorem~\ref{theorem measure S(t+h)-S(t)}. His proof borrows the setup in \cite{Tsa86}, but the main engine used is the resonance method (which was developed more recently) rather than Tsang's original method of high moments, which is the basis of our work.

The rest of this article is organized as follows. In \S\ref{section: reduction}, we show that Theorem~\ref{theorem 1} is a direct consequence of Theorem~\ref{theorem measure S(t+h)-S(t)} and set up the proof of Theorem~\ref{theorem measure S(t+h)-S(t)}. Then, in \S\ref{section: moments of W(t)}, \S\ref{section: moments of S(t+h)-S(t)} and \S\ref{section: moments of R(t)}, we make explicit Tsang's proof of \cite[Theorem 3]{Tsa86}. In \S\ref{section: completing the proof}, we complete the proof of Theorem~\ref{theorem measure S(t+h)-S(t)} and also justify the explicit constants stated in Theorem~\ref{theorem 2}.

Throughout we assume that $T$ is sufficiently large and $c$ is some positive absolute constant. Their precise values are not specified and may vary at each occurrence. When we write the big-$O$ and little-$o$ notations, they are with respect to $T$ as $T\to \infty$ unless otherwise indicated.

\section{Reduction to Theorem~\ref{theorem measure S(t+h)-S(t)} and setup for the proof}\label{section: reduction}

\subsection{From $S(t+h)-S(t)$ to zero gaps}\label{subsection 2.1}

We first show how to deduce Theorem~\ref{theorem 1} (without explicitly computing the constants in Theorem~\ref{theorem 2}) from Theorem~\ref{theorem measure S(t+h)-S(t)}. Let us begin with a simple lemma.
\begin{lemma}\label{counting lemma}
    Let $\{a_n\}_{1\leq n\leq N}$ be a set of $N$ points in an interval $[a,b]$ and $I\subset [a,b]$. Further let $h>0$ and $k\in \bb{N}$. If $\#\{a_n\in [t,t+h]\}\geq k$ for all $t\in I$, then $N\geq k(\meas(I)/h-1)$.  
\end{lemma}

\begin{proof}
    Call a positive integer $n$ \textit{good} if there exist $n$ points $t_1<t_2<\ldots <t_n$ in $I$ such that the intervals $[t_i,t_i+h)$ are pairwise disjoint and all lie in $[a,b]$. It can easily be seen that if $\meas(I)>nh$, then $n$ is good. Let $m$ be the maximal good integer. The above observation gives $\meas(I)\leq (m+1)h$, or $m\geq \meas(I)/h-1$. The conclusion follows readily.
\end{proof}

Now we prove Theorem~\ref{theorem 1}. For $0\leq h\leq 1$, we find from \eqref{riemann-von mangoldt} that
\begin{equation}\label{N(t+h)-N(t)}
    N(t+h)-N(t)=\frac{h}{2\pi}\log\frac{t}{2\pi e}+S(t+h)-S(t)+O(h+t^{-1}).
\end{equation}
We first deal with short gaps. Set 
\[
h=\frac{2\pi r(1-\theta r^{-\alpha})}{\log T}
\]
where $\theta>0$ is to be chosen, and let 
\[
I^+:=\left\{t\in [T,2T]: S(t+h)-S(t)\geq c(h\log t)^{1-\alpha}\right\}
\]
for some $0<c\leq C$ where $C$ is the constant from Theorem~\ref{theorem measure S(t+h)-S(t)}. For $t\in I^+$, \eqref{N(t+h)-N(t)} implies that
\begin{align*}
    N(t+h)-N(t)\geq& r\left(1-\frac{\theta}{r^\alpha}\right)+c(2\pi r)^{1-\alpha} \left(1-\frac{\theta}{r^\alpha}\right)^{1-\alpha}+O(h)\\
    \geq& r+A_1(c,\theta)r^{1-\alpha}+O(h)
\end{align*}
where
\begin{equation}\label{def A_1}
    A_1(c,\theta):=c\left[2\pi\left(1-\frac{\theta}{r^\alpha}\right)\right]^{1-\alpha}-\theta.
\end{equation}
For sufficiently small $\theta>0$ this term is strictly positive, and since $N(t+h)-N(t)$ is always an integer, for any $t\in I^+$ we have
\[
\#\left\{\gamma_n\in [t,t+h]: \gamma_{n+r}-\gamma_n\leq h\right\} \geq \lceil A_1(c,\theta)r^{1-\alpha}\rceil\geq 1
\]
(we removed the additional $O(h)$ term since we may assume that $A_1(c,\theta)r^{1-\alpha}$ is not an integer). Lemma~\ref{counting lemma} therefore gives
\[
\#\{\gamma_n\in [T,2T]: \gamma_{n+r}-\gamma_n\leq h\}\geq \left(\frac{\meas(I^+)}{h}-1\right)\lceil A_1(c,\theta)r^{1-\alpha}\rceil,
\]
and we see from the definition \eqref{def D} of $D$ that
\begin{equation}\label{small gap proportion}
    D(1-\theta r^{-\alpha},r,T)\geq \frac{\meas(I^+)}{T} \frac{\lceil A_1(c,\theta)r^{1-\alpha}\rceil}{r(1-\theta r^{-\alpha})}\left(1+O\left(\frac{1}{T\log T}\right)\right).
\end{equation}
Theorem~\ref{theorem measure S(t+h)-S(t)} guarantees that $\meas(I^+)\gg T$ for fixed $r$, and hence the right-hand side is strictly positive as $T\to \infty$. This shows $\wt{\mu}_r\leq 1-\theta r^{-\alpha}$, as desired.
    
If we set
\[
h=\frac{2\pi r(1+\theta r^{-\alpha})}{\log T}
\]
and
\[
I^-:=\left\{t\in [T,2T]: S(t+h)-S(t)\leq -c(h\log t)^{1-\alpha}\right\},
\]
a similar argument shows that there exist large gaps of size at least $h$ as long as
\[
A_2(c,\theta):=c\left[2\pi\left(1+\frac{\theta}{r^\alpha}\right)\right]^{1-\alpha}-\theta>0.
\]
However, we can no longer apply Lemma~\ref{counting lemma} to bound the proportion of large gaps from below, as was done in the previous case. Indeed, the condition $N(t+h)-N(t)<r$ guarantees only one $r$-gap of size $\geq h$, and this gap can possibly be much longer than the scale of $h$.

To remedy this, we use a Cauchy--Schwarz argument originally due to Heath-Brown \cite{Tit86} that Lukas Kornmann kindly pointed out. The argument also appears in his forthcoming work \cite{Kornmann26}. Starting with the observation that
\[
I^- \subset [T,\gamma_0]\bigcup_{\substack{\gamma_n\in [T,2T]\\ \gamma_{n+r}-\gamma_n\geq h}} [\gamma_n,\gamma_{n+1}]
\]
where $\gamma_0=\min\{\gamma: \gamma\geq T\}$, we have
\begin{align*}
    \meas(I^-)\leq& O(1)+\sum_{\substack{\gamma_n\in [T,2T]\\\gamma_{n+r}-\gamma_n\geq h}} (\gamma_{n+1}-\gamma_n)\\
    \leq& O(1)+\sqrt{\#\{\gamma_n\in [T,2T]: \gamma_{n+r}-\gamma_n\geq h\}\cdot \sum_{\gamma_n\in [T,2T]}(\gamma_{n+1}-\gamma_n)^2}.
\end{align*}
After rearranging, this implies that
\[
D(1+\theta r^{-\alpha},r,T)\leq 1-\left(\frac{\meas(I^-)-O(1)}{T}\right)^2\left(\frac{1}{2\pi T/\log T}\sum_{\gamma_n\in [T,2T]}(\gamma_{n+1}-\gamma_n)^2\right)^{-1}.
\]
For fixed $r$, we have $\meas(I^-)\gg T$, while the term in the second parenthesis remains bounded as $T\to \infty$. In fact, by a result of Korol\"{e}v \cite[Theorem 3]{Kor08}, this term can be explicitly bounded by
\begin{equation}\label{Korolev bound}
    \frac{1}{2\pi T/\log T} \sum_{\gamma_n\in [T,2T]}(\gamma_{n+1}-\gamma_n)^2<4\pi e^{99.8}.
\end{equation}
Under RH, Fujii \cite[Theorem 1]{Fuj90} obtained a much sharper bound $8.55$ in place of $4\pi e^{99.8}$.

\subsection{Setting up the proof of Theorem~\ref{theorem measure S(t+h)-S(t)}}

Throughout we assume that $h\in [\frac{1}{\log T},\frac{1}{\log\log T}]$, although in our applications we choose $h\approx \frac{2\pi r}{\log T}$, the average length of an $r$-gap. Tsang's proof of \cite[Theorem 3]{Tsa86}, similar to his celebrated $\Omega$-result on $S(t)$ in the same paper, starts with a convolution formula: for each $t\in [T,2T]$,
\[
\int_{-\infty}^\infty (S(t+h+u)-S(t+u))K(u) \md u \approx W(t)+R(t)
\]
where $W(t)$ is a Dirichlet polynomial, $R(t)$ is a sum over zeta zeros off the critical line (which vanishes under RH), and $K(u)$ is a test function satisfying a certain growth condition. Tsang's key idea is to extract large values of $\pm(W(t)+R(t))$ by comparing high moments of $W(t)$ and $R(t)$. We will study the measure of subsets of $[T,2T]$ where $W(t)+R(t)$ attains large values, which are then translated to $S(t+h)-S(t)$ via the convolution formula.

To be precise, for some parameter $\tau>0$ that will be chosen as a suitable multiple of $h^{-1}$, we put
\[
K_\tau(z)=\frac{1}{2\pi\tau}\left(\frac{\sin(\tau z/2)}{z/2}\right)^2
\]
so that
\[
\int_{-\infty}^\infty  K_\tau(u)\md u=1
\]
and
\[
\wh{K_\tau}(\xi)=
\begin{cases}
    1-2\pi|\xi|/\tau,& |\xi|\leq \frac{\tau}{2\pi},\\
    0,& \text{otherwise}.
\end{cases}
\]
Here $\wh{f}(\xi):=\int_{-\infty}^\infty f(u)e^{-2\pi i\xi u}\md u$ denotes the Fourier transform of $f$. By \cite[Lemma 5]{Tsa86}, for any $t>0$,
\begin{align*}
    \int_{-\infty}^\infty &(S(t+h/2+u)-S(t-h/2+u)) K_\tau(u)\md u\\
    =&-\frac{2}{\pi}\sum_{n\leq e^{\tau}} \frac{\Lambda(n)}{\sqrt{n}\log n}\left(1-\frac{\log n}{\tau}\right)\sin\left(\frac{h}{2}\log n\right)\cos(t\log n)\\
    &\hspace{0.5cm}+2\sum_{\beta>1/2}\int_0^{\beta-1/2}\Im\left\{K_\tau(\gamma-t-h/2-i\sigma)-K_\tau(\gamma-t+h/2-i\sigma)\right\}\md \sigma+O(t^{-1}).
\end{align*}
Let $t\in [T,2T]$. Since $S(t)=O(\log t)$ and $K_\tau(u)\ll \tau^{-1}u^{-2}$ for large $u$,
\[
\int_{|u|\geq \log T}S(t\pm h/2+u)K_\tau(u)\md u=O(\tau^{-1}).
\]
It follows that
\begin{equation}\label{convolution}
    \int_{-\log T}^{\log T} (S(t+h/2+u)-S(t-h/2+u)) K_\tau(u)\md u=W(t)+R(t)+O(h+\tau^{-1})
\end{equation}
where
\begin{equation}\label{def W(t)}
    W(t):=-\frac{2}{\pi}\Re \sum_{\substack{\ell\in\{1,2\} \\ p\leq e^{\tau/\ell}}}\left(1-\frac{\ell \log p}{\tau}\right)\frac{\sin(\frac{h\ell}{2}\log p)}{\ell p^{\ell(1/2+it)}}
\end{equation}
and
\[
R(t):=2\sum_{\beta>1/2}\int_0^{\beta-1/2}\Im\left\{K_\tau(\gamma-t-h/2-i\sigma)-K_\tau(\gamma-t+h/2-i\sigma)\right\}\md \sigma.
\]

The key step is the next lemma whose first part is a restatement of \cite[Lemma 4]{Tsa86}. Here $W(t)$ and $R(t)$ can be arbitrary real-valued functions. 
\begin{lemma}\label{key lemma}
    Suppose that there exist $k\in \bb{N}$ and $M_1,M_2,M_3,M_4>0$ such that
    \begin{enumerate}[label=(\roman*)]
        \item 
        \[
        \int_T^{2T}W(t)^{2k}\md t\geq TM_1^{2k}, \qquad \int_T^{2T}W(t)^{4k+2}\md t\leq TM_2^{4k+2};
        \]
        \item 
        \[
        \left|\int_T^{2T} W(t)^{2k+1}\md t\right|\leq \eps TM_1^{2k+1}, \qquad 0\leq \eps<1;
        \]
        \item 
        \[
        \int_T^{2T}|R(t)|^{2k+1}\md t\leq TM_3^{2k+1},\qquad \int_T^{2T}R(t)^{4k+2}\md t\leq TM_4^{4k+2};
        \]
        \item
        \[
        M_0:=\left(\frac{1-\eps}{2}\right)^{\frac{1}{2k+1}}M_1- M_3\geq 0.
        \]
    \end{enumerate}
    Then
    \[
    \max_{t\in [T,2T]}\pm(W(t)+R(t))\geq M_0.
    \]
    Moreover, for any $0<M<M_0$, the measure of the sets 
    \[
    J_M^{\pm}:=\{t\in [T,2T]: \pm(W(t)+R(t))\geq M\}
    \]
    satisfies
    \[
    \frac{\meas(J_M^{\pm})}{T}\geq \left(1-\frac{M}{M_0}\right)^2\left(\frac{\frac{1-\eps}{2}M_1^{2k+1}-M_3^{2k+1}}{M_2^{2k+1}+M_4^{2k+1}}\right)^2.
    \]
\end{lemma}

\begin{proof}
    Let $W_+(t)=\max\{W(t),0\}$ and $W_{-}(t)=\min\{W(t),0\}$. Following the proof of \cite[Lemma 4]{Tsa86}, we find that
    \[
    \int_T^{2T}W_+(t)^{2k+1}\md t= TM_5^{2k+1}
    \]
    for some
    \[
    M_5\geq \left(\frac{1-\eps}{2}\right)^{\frac{1}{2k+1}}M_1.
    \]
    Then
    \begin{align*}
        T(M_5^{2k+1}-M_3^{2k+1})\leq& \int_T^{2T}\left(W_+(t)^{2k+1}-|R(t)|^{2k+1}\right) \md t\\
        =&\int_T^{2T} \left(W_+(t)-|R(t)|\right)\left(\sum_{i=1}^{2k+1}W_+(t)^{2k+1-i}|R(t)|^{i-1}\right)\md t.
    \end{align*}
    For any $t\not \in J_M^+$, $W_+(t)-|R(t)|\leq \max\{-|R(t)|, W(t)+R(t)\}<M$. By H\"{o}lder's inequality and the Cauchy--Schwarz inequality, the above integral is at most
    \begin{align*}
        \leq& MT \left(\sum_{i=1}^{2k+1}M_5^{2k+1-i} M_3^{i-1}\right)+\sqrt{\meas(J_M^+)}\sqrt{\int_T^{2T}\left(W_+(t)^{2k+1}-|R(t)|^{2k+1}\right)^2 \md t}\\
        \leq& MT\frac{M_5^{2k+1}-M_3^{2k+1}}{M_5-M_3}+\sqrt{\meas(J_M^+)}\left(\sqrt{\int_T^{2T}W(t)^{4k+2}\md t}+\sqrt{\int_T^{2T} R(t)^{4k+2}\md t}\right),
    \end{align*}
    where we applied Minkowski's inequality for $L_2$ norm in the last step. We find after a little rearranging that for all $0<M<M_5-M_3$,
    \[
    \frac{\meas(J_M^+)}{T}\geq \left(1-\frac{M}{M_5-M_3}\right)^2\left(\frac{M_5^{2k+1}-M_3^{2k+1}}{M_2^{2k+1}+M_4^{2k+1}}\right)^2.
    \]
    The same argument works for $J_M^-$ as well after replacing $W_+$ by $W_-$.
\end{proof}

Next, supposing that the conditions in the preceding lemma are already verified, we pass from extreme values of $W(t)+R(t)$ to those of $S(t+h)-S(t)$. A similar procedure is carried out in \cite[\S{6}]{Dob24} for $S(t)$. 

Let $M_i=M_i(h,\tau,k,T)$ be as in Lemma~\ref{key lemma}. For brevity we will drop the parameters from the notation. For any $0<M<M_0$, \eqref{convolution} gives
\begin{align*}
    \meas(J_M^+)M \leq& \int_{J_M^+} W(t)+R(t) \md t\\
    \leq & \int_{J_M^+} \int_{-\log T}^{\log T} (S(t+h/2+u)-S(t-h/2+u))K_\tau(u) \md u \md t\\
    &\hspace{7cm}+O(\meas(J_M^+)(h+\tau^{-1}))\\
    =&\int_{-\infty}^\infty (S(u+h/2)-S(u-h/2))\cdot (\1_{J_M^+}*K_\tau\cdot \1_{[-\log T,\log T]})(u) \md u\\
    &\hspace{7cm}+O(\meas(J_M^+)(h+\tau^{-1})).
\end{align*}
Let $0<V<M$ and 
\[
I_{V}^{\pm}:=\{t\in [T-\log T,2T+\log T]: \pm(S(t+h/2)-S(t-h/2))\geq V\}.
\]
Note that 
\[
\left|(\1_{J_M^+}*K_\tau\cdot \1_{[-\log T,\log T]})(u)\right|\leq \int_{-\infty}^\infty K_\tau(y)\md y=1
\]
for any $u$ and
\[
\int_{-\infty}^\infty  (\1_{J_M^+}*K_\tau\cdot \1_{[-\log T,\log T]})(u)\md u\leq \meas(J_M^+).
\]
Moreover, the support of the convolution is contained in $[T-\log T, 2T+\log T]$. Hence, again by Cauchy--Schwarz,
\begin{multline*}
    \meas(J_M^+)M\leq \meas(J_M^+)(V+O(h+\tau^{-1}))\\
    +\sqrt{\meas(I_{V}^+)}\sqrt{\int_{T-\log T}^{2T+\log T}|S(u+h/2)-S(u-h/2)|^2\md u}.
\end{multline*}
If
\begin{equation}\label{square S(t+h)-S(t)}
    \frac{1}{T}\int_{T-\log T}^{2T+\log T}|S(t+h)-S(t)|^2\md t < A_0=A_0(h,T),
\end{equation}
then for any $V<M<M_0$,
\begin{align*}
    \frac{\meas(I_{V}^{+})}{T}\geq& \left(\frac{\meas(J_M^+)}{T}\right)^2\frac{(M-V+O(h+\tau^{-1}))^2}{A_0}\\
    \geq& \left(1-\frac{M}{M_0}\right)^4\frac{(M-V+O(h+\tau^{-1}))^2}{A_0}\left(\frac{\frac{1-\eps}{2}M_1^{2k+1}-M_3^{2k+1}}{M_2^{2k+1}+M_4^{2k+1}}\right)^4
\end{align*}
where we applied Lemma~\ref{key lemma}. The maximum of the last expression occurs at $M=(M_0+2V)/3$, and thus for any $0<V<M_0$,
\begin{equation}\label{measure I_V}
    \frac{\meas(I_{V}^{+})}{T}\geq \frac{2^4}{3^6 A_0}\frac{(M_0-V+O(h+\tau^{-1}))^6}{M_0^4}\left(\frac{\frac{1-\eps}{2}M_1^{2k+1}-M_3^{2k+1}}{M_2^{2k+1}+M_4^{2k+1}}\right)^4.
\end{equation}
The same lower bound holds for $\meas(I_{V}^{-})$ as well. The proof of Theorem~\ref{theorem measure S(t+h)-S(t)} will be completed in \S\ref{section: completing the proof} after we determine the sizes of $M_i, A_0$ for $h\in [\frac{1}{\log T},\frac{1}{\log\log T}]$ and maximize $M_0$ as well as the right-hand side of \eqref{measure I_V} over all admissible $\tau$ and $k$. 

\subsection{Back to $r$-gaps}\label{section: r-gaps} 

We now return to the argument in \S\ref{subsection 2.1} and show how to explicitly estimate the sizes and proportions of extreme $r$-gaps (as required for  Theorem~\ref{theorem 2}) in terms of $M_i$ and $A_0$. For some parameters $\theta,c>0$, let $h=(1-\theta r^{-\alpha})\cdot 2\pi r/\log T$ and $V=c(h\log T)^{1-\alpha}$. On the one hand, we need $V<M_0$, i.e., 
\[
c<c^-_1:=\frac{M_0}{[2\pi r(1-\theta r^{-\alpha})]^{1-\alpha}}.
\]
On the other hand, we need $A_1(c,\theta)>0$, that is,
\[
c>c^-_2:=\frac{\theta}{[2\pi(1-\theta r^{-\alpha})]^{1-\alpha}}.
\]
This is possible if $0<\theta<M_0/r^{1-\alpha}$. We will show the existence of some absolute constant $\Theta>0$ such that given any $0<\theta<\Theta$, $\sup_{k,\tau} M_0\geq \Theta r^{1-\alpha}$ uniformly for all $r\in \bb{N}$ as $T\to \infty$. Consequently, for each $r$ there exist $k,\tau$ such that $c_2^-<c_1^-$ whenever $0<\theta<\Theta$. In view of \eqref{small gap proportion} and \eqref{measure I_V}, for all such $\theta$ we have
\[
D^-(1-\theta r^{-\alpha},r)\geq 
\liminf_{T\to \infty}\sup_{\substack{c,k,\tau\\ M_0/r^{1-\alpha}>\theta\\ c^-_2<c <c^-_1}}f_1(\theta,r,c,k,\tau)
\]
where
\begin{multline}\label{proportion f1}
    f_1(\theta,r,c,k,\tau):=\frac{2^4\big(M_0-c [2\pi r(1-\theta r^{-\alpha})]^{1-\alpha}\big)^6}{3^6 A_0 M_0^4} \left(\frac{\frac{1-\eps}{2}M_1^{2k+1}-M_3^{2k+1}}{M_2^{2k+1}+M_4^{2k+1}}\right)^4\\
    \times 
    \frac{\lceil A_1(c,\theta)r^{1-\alpha}\rceil}{r(1-\theta r^{-\alpha})}.
\end{multline}

By the same argument, for $h=(1+\theta r^{-\alpha})\cdot 2\pi r/\log T$, there exists $\vartheta>0$ such that whenever $0<\theta<\vartheta$, $\sup_{k,\tau}M_0\geq \vartheta r^{1-\alpha}$ uniformly for all $r$. Unlike the previous case, recall from \S\ref{subsection 2.1} that the size of $\lceil A_2(c,\theta)\rceil$ does not affect the lower bound on proportions of large gaps as long as it is positive, and hence we can take $c$ to be arbitrarily close to 
\[
\frac{\theta}{[2\pi (1+\theta r^{-\alpha})]^{1-\alpha}}.
\]
As a result,
\[
D^+(1+\theta r^{-\alpha},r)\leq 
1-\liminf_{T\to \infty} \sup_{\substack{k,\tau\\ M_0/r^{1-\alpha}>\theta}} f_2(\theta,r,k,\tau)
\]
where
\begin{equation}\label{proportion f2}
    f_2(\theta,r,k,\tau):=\frac{2^8\big(M_0-\theta r^{1-\alpha}\big)^{12}}{B_0\cdot 3^{12} A_0^2 M_0^8} \left(\frac{\frac{1-\eps}{2}M_1^{2k+1}-M_3^{2k+1}}{M_2^{2k+1}+M_4^{2k+1}}\right)^8,
\end{equation}
\[
B_0=
\begin{cases}
    4\pi e^{99.8} & \text{unconditionally};\\
    8.55 & \text{under RH}.
\end{cases}
\]

When $h\ll \frac{1}{\log T}$ and $T\to \infty$, we will see that $k$ can be chosen to be $O(1)$ and all of $M_i$ and $A_0$ are also $O(1)$.  In particular, when $\theta<\Theta$ and $\theta<\vartheta$, respectively, $f_1(\theta,r,c,k,\tau)$ and $f_2(\theta,r,k,\tau)$ are strictly positive for appropriate choices of $c,k,\tau$. It thus follows from the definition that
\[
\wt{\lambda}_r\geq 1+\frac{\Theta}{r^\alpha} \qquad \text{and} \qquad  \wt{\mu}_r\leq 1-\frac{\vartheta}{r^\alpha}.
\]

The next few sections are dedicated to studying moments of $W(t)$ and $R(t)$ so that we can determine the values of $M_i$ (which are really expressions in $\tau$, $k$, $h$ and $T$). Along the way we need bounds on moments of $S(t+h)-S(t)$, which will give the value of $A_0$ for free. We will follow the original argument in \cite{Tsa86} and make use of some of the explicit results proved in \cite{STT22}.

\section{Moments of $W(t)$}\label{section: moments of W(t)}
We start with two lemmas on moments of Dirichlet polynomials. In what follows, $\{a_p\}_{p\leq x}$ where $x\geq 2$ is an arbitrary sequence of complex numbers indexed by primes and $f(t)$ is either the real or imaginary part of $\sum_{p\leq x} a_p p^{-it}$.

\begin{lemma}\label{lemma moments of f(t)}
For any positive integer $k$,
\begin{equation}\label{int f(t)^{2k} 1}
    \left|\int_T^{2T} f(t)^{2k}\md t-\frac{T}{4^k}\binom{2k}{k} \sum_{\textbf{p}}|a_\textbf{p}|^2 |\textbf{p}|\right|\leq B_1(k) \left(\sum_{p\leq x} p|a_p|^2\right)^k
\end{equation}
and
\begin{equation}\label{int f(t)^{2k+1}}
\left|\int_T^{2T} f(t)^{2k+1}\md t\right|\leq B_2(k) \left(\sum_{p\leq x} p|a_p|^2\right)^{k+1/2} 
\end{equation}
where 
\[
B_1(k):=\frac{3\pi m_0}{4^k}\sum_{m=0}^{2k}\binom{2k}{m}\sqrt{m!(2k-m)!},
\]
\[
B_2(k):=\frac{3\pi m_0}{2^{2k+1}}\sum_{m=0}^{2k+1}\binom{2k+1}{m}\sqrt{m!(2k+1-m)!},
\]
$m_0=\sqrt{1+\frac{2}{3}\sqrt{\frac{6}{5}}}$, $\textbf{p}$ denotes a prime $k$-tuple $(p_1,\ldots,p_k)$, $p_i\leq x$, $a_{\textbf{p}}=a_{p_1}\ldots a_{p_{k}}$, and $|\textbf{p}|$ is the number of permutations of $(p_1,\ldots,p_k)$. In particular,
\begin{equation}\label{int f(t)^{2k} 2}
   \int_T^{2T} f(t)^{2k}\md t\leq  T\frac{(2k)!}{4^k k!} \left(\sum_{p\leq x}|a_p|^2\right)^k + B_1(k) \left(\sum_{p\leq x} p|a_p|^2\right)^k. 
\end{equation}
\end{lemma}

\begin{proof}
    This is essentially \cite[Lemma 2]{Tsa86}, except that the constant $m_0$ (whose precise value will not matter in our argument) comes from an improvement of Preissmann \cite{Pre84} on Hilbert's inequality. See also \cite[\S{4.2}]{STT22}. The estimate \eqref{int f(t)^{2k} 2} follows from \eqref{int f(t)^{2k} 1} with the trivial observation that $|\textbf{p}|\leq k!$ for any prime $k$-tuple.
\end{proof}

Using Stirling's formula, and more precisely \eqref{(2k)!/(k!) inequality}, we find that both $B_1(k)$ and $B_2(k)$ are $O(k^k)$.

\begin{lemma}\label{lemma lower bound moment f(t)}
    If $|a_p|$ is non-increasing in $p$ for $p\geq N_0\geq 2$ and $k$ is a positive integer with $\ell(N_0,k)<x$ where $\ell(N_0,k):=2(\pi(N_0)+k)\log(\pi(N_0)+k)$,
    then 
    \[
    \int_T^{2T} f(t)^{2k}\md t\geq T\frac{(2k)!}{4^k k!}\left(\sum_{\ell(N_0,k)< p\leq x}|a_p|^2\right)^k-B_1(k) \left(\sum_{p\leq x} p|a_p|^2\right)^k.
    \]

\end{lemma}

\begin{proof}
By considering only the $k$-tuples $\textbf{p}=(p_1,\ldots, p_k)$ where the $p_i$'s are distinct, we see that
    \begin{align*} \sum_{\textbf{p}}|a_{\textbf{p}}|^2|\textbf{p}|\geq k! \sum_{p_1\leq x}|a_{p_1}|^2\left(\sum_{\substack{p_2\neq p_1\\p_2\leq x}}\left(|a_{p_2}|^2\ldots \left(\sum_{\substack{p_k\neq p_1,\ldots,p_{k-1}\\p_k\leq x}}|a_{p_k}|^2\right)\ldots\right)\right).
    \end{align*}
    Note \footnote{This follows from the fact that the $n$th prime is at most $n(\log n+\log\log n)$ for all $n\geq 6$ (see \cite[Corollary to Theorem 3]{RS62}).} that there are at least $n$ prime numbers $\leq 2n\log n$ for all $n\geq 2$, so there are at least $k$ primes in $[N_0,\ell(N_0,k)]$, and since $|a_p|$ is non-increasing in $p$ for all $p\geq N_0$, 
    \[
    \sum_{\substack{p_k\neq p_1,\ldots,p_{k-1}\\p_k\leq x}}|a_{p_k}|^2\geq \sum_{\ell(N_0,k)< p\leq x}|a_p|^2.
    \]
    Repeating the same reasoning leads to
    \[\sum_{\textbf{p}}|a_{\textbf{p}}|^2|\textbf{p}|\geq k! \left(\sum_{\ell(N_0,k)<p\leq x}|a_p|^2\right)^k.
    \]
    Applying \eqref{int f(t)^{2k} 1} finishes the proof.

\end{proof}

For $W(t)$ as defined in \eqref{def W(t)}, write
\[
W(t)=W_1(t)+W_2(t),
\]
\[
W_\ell(t):=-\frac{2}{\pi}\Re \sum_{p\leq e^{\tau/\ell}}\left(1-\frac{\ell \log p}{\tau}\right)\frac{\sin(\frac{h\ell}{2}\log p)}{\ell p^{\ell(1/2+it)}}.
\]
Denote the $L_p$ norm of $f$ by
\[
\|f\|_{p}:=\left(\int_T^{2T}|f(t)|^{p}\md t\right)^{1/p}.
\]
Then, by Minkowski's inequality,
\[
\left(\|W_1\|_{2k}-\|W_2\|_{2k}\right)^{2k}\leq \int_{T}^{2T}W(t)^{2k}\md t\leq \left(\|W_1\|_{2k}+\|W_2\|_{2k}\right)^{2k},
\]
and further by the triangle inequality and Cauchy--Schwarz,
\begin{align*}
    \left|\int_{T}^{2T}W(t)^{2k+1}\md t\right|\leq& \left|\int_{T}^{2T}W_1(t)^{2k+1}\md t\right|+\left|\sum_{j=0}^{2k}\binom{2k+1}{j}\int_{T}^{2T}W_1(t)^{j}W_2(t)^{2k+1-j}\md t\right|\\
    \leq& \left|\int_{T}^{2T}W_1(t)^{2k+1}\md t\right|+\binom{2k+1}{k}\sum_{j=0}^{2k}\|W_1\|_{2j}^{j}\|W_2\|_{4k+2-2j}^{2k+1-j}.
\end{align*}
Applying Lemma~\ref{lemma moments of f(t)} with $f(t)=W_2(t)$, 
\[
a_p=-\frac{2}{\pi}\left(1-\frac{2\log p}{\tau}\right)\frac{\sin(h\log p)}{2p^{1+it}},\qquad x=e^{\tau/2},
\]
and using the prime number theorem as well as Stirling's estimate, we deduce that
\[
\int_T^{2T}W_2(t)^{2k}\md t\ll T(ck)^kh^{2k}+O\left((ck)^k(\tau h)^{2k}\right).
\]
Since $h$ is small, this will allow us to essentially omit the contributions from $W_2$ to $W$. 

Now let $\tau\asymp 1/h$, $\tau h<2\pi$ and $\log k\leq \tau/2$. A calculus exercise shows that $\sin(\frac{h}{2}\log p)p^{-1/2}$ decreases in $p$ for $e^2<p<e^{2\pi/h}$. Another application of Lemmas~\ref{lemma moments of f(t)} and \ref{lemma lower bound moment f(t)} with $f(t)=W_1(t)$,
\[
a_p=-\frac{2}{\pi}\left(1-\frac{\log p}{\tau}\right)\frac{\sin(\frac{h}{2}\log p)}{p^{1/2}},\qquad x=e^{\tau}
\]
(note that $|a_p|$ is decreasing in $p$ for $e^2<p<x$) gives
\begin{align*}
    \int_T^{2T}W(t)^{2k}\md t\geq TM_1^{2k}
\end{align*}
where 
\begin{equation}\label{M1}
    M_1=\frac{(1+o(1))}{\pi}\left(\frac{(2k)!}{k!}\right)^{\frac{1}{2k}}\sqrt{\int_{h\log k}^{\tau h}\left(1-\frac{u}{\tau h}\right)^2\frac{\sin^2(u/2)}{u}\md u}
\end{equation}
(since $\tau h\asymp 1$ and we only care about the asymptotic behavior as $T\to \infty$, here the lower bound $h\log k$ can be replaced by 0 as long as $h\log k=o(1)$) and
\[
\int_T^{2T}W(t)^{4k+2}\leq TM_2^{4k+2} 
\]
where
\begin{equation}\label{M2}
    M_2=\frac{(1+o(1))}{\pi}\left(\frac{(4k+2)!}{(2k+1)!}\right)^{\frac{1}{4k+2}}\sqrt{\int_0^{\tau h}\left(1-\frac{u}{\tau h}\right)^2\frac{\sin^2(u/2)}{u}\md u}.
\end{equation}
On the other hand, if we further assume that $\tau\leq \frac{\log T}{k+1/2}-\log\log T$, then by \eqref{int f(t)^{2k+1}},
\[
\int_T^{2T}W(t)^{2k+1}\md t\ll e^{\tau(k+1/2)}\leq \frac{T}{(\log T)^{k+1/2}}=o(TM_1^{2k+1}).
\]
Hence, Lemma~\ref{key lemma}(ii) holds with $\eps=o(1)$.

We now turn to the more difficult term $R(t)$.

\section{Moments of $S(t+h)-S(t)$}\label{section: moments of S(t+h)-S(t)}
To study moments of the zero sum $R(t)$, we first need some estimates for moments of $S(t+h)-S(t)$. Throughout this section $v$ stands for a positive integer. Moreover, we shall use the notation $N(\sigma,T)$ to denote the number of zeta zeros with $\beta\geq \sigma$ and $0< \gamma\leq T$, and
\[
N(\sigma;T_1,T_2):=N(\sigma,T_2)-N(\sigma,T_1).
\]

\begin{lemma}\label{lemma S(t) dirich poly}
    Let $0<\eps\leq 1/88$. If there exist $\sigma_0>1/2$ and $L>0$ such that 
    \begin{equation}\label{zero density}
        N(\sigma; T,2T)\leq L\cdot T^{1-(\sigma-1/2)/4}\log T
    \end{equation}
    for all $\sigma\in [1/2,\sigma_0]$ and all sufficiently large $T$, then
    \[
    \frac{1}{T}\int_T^{2T}\left|S(t)+\frac{1}{\pi}\sum_{p\leq T^{3\eps/v}}\frac{\sin(t\log p)}{\sqrt{p}}\right|^{2v}\md t\leq C(\eps,v)+o(1)
    \]
    where
    \[
    C(\eps,v):=\frac{1}{6}\left(1+\sum_{n=1}^4 R_n(\eps,v)\right)\left(\frac{12 a_2 a_4(\eps,v)}{\eps}v\right)^{2v}
    \]
    with
    \[
    R_1(\eps,v):=\frac{8a_0L}{\eps}(8\eps)^{2v}\frac{(2v)!}{v^{2v-1}}+\left(\frac{3\eps}{2\pi a_2 a_4(\eps,v)\sqrt{2v}}\right)^{2v}(13+18^{-v}),
    \]
    \[
    R_2(\eps,v):=\sqrt{13}\left(\frac{(12 e)^2a_1\eps}{a_2a_4(\eps,v)\sqrt{v}}\right)^{2v}\sqrt{1+\frac{8a_0L}{\eps}\left(\frac{8\eps}{e}\right)^{8v}\frac{(8v)!}{v^{8v-1}}},
    \]
    \[
    R_3(\eps,v):=\sqrt{13}\left(\frac{24 e^2\eps}{\pi a_2a_4(\eps,v)\sqrt{v}}\right)^{2v}\sqrt{1+\frac{8a_0L}{\eps}\left(\frac{8\eps}{e^2}\right)^{4v}\frac{(4v)!}{v^{4v-1}}},
    \]
    \[
    R_4(\eps,v):=\sqrt{13}\left(\frac{6a_1\eps}{a_2a_4(\eps,v)\sqrt{v}}\right)^{2v}\sqrt{1+\frac{8a_0L}{\eps}(8\eps)^{4v}\frac{(4v)!}{v^{4v-1}}},
    \]
    and
    \[
    a_0:=1.5453,
    \]
    \[
    a_1:=13+\frac{13}{5\pi}+\frac{13}{3\pi e},
    \]
    \[
    a_2:=\frac{13}{2}+\frac{9}{5\pi}+\frac{13}{6\pi e},
    \]
    \[
    a_3:=\frac{3\pi a_1}{2}+\frac{139}{75},
    \]
    \[
    a_4(\eps,v):=1+\frac{\eps a_3}{\pi v a_2}.
    \]
    In particular, when $v\to \infty$,
    \[
    C(\eps,v)\sim \left(\frac{12 a_2 a_4(\eps,v)}{\eps}v\right)^{2v}.
    \]
\end{lemma}
\begin{proof}
    This is \cite[Theorem 5]{STT22} except that we do not necessarily have to require $\sigma_0=1$ in our assumption, which is clear once the reader examines the proof of \cite[Lemma 3]{STT22} where the zero-density estimate of the shape \eqref{zero density} comes into place. Roughly speaking, this is because the right-hand side of \eqref{zero density} decreases sufficiently rapidly in $\sigma$ to the right of the half-line.
\end{proof}

\begin{lemma}\label{lemma zero density}
    The inequality \eqref{zero density} holds with $L=3.276$ and $\sigma_0=0.549$. It also holds with $L=642.86$ and $\sigma_0=1$.
\end{lemma}
\begin{proof}
    See \cite[Theorem 1]{Sim20} for the first part. The second part is \cite[Lemma 1]{STT22}.
\end{proof}

The main result in this section is the following:
\begin{proposition}\label{prop: moment S(t+h)-S(t)}
    For $0\leq h\leq 1$ and $v\leq \frac{\log T}{3\log\log T}$,
    \[
    \frac{1}{T}\int_{T}^{2T}\left|S(t+h)-S(t)\right|^{2v} \md t\leq \wt{C}(h,v)^{2v} 
    \]
    where
    \begin{equation}\label{def tilde C}
        \wt{C}(h,v):=
        \begin{cases}
            \displaystyle 2C(1/88,v)^{\frac{1}{2v}}+\frac{\sqrt{2.021}}{\sqrt{2}\pi}\left(\frac{(2v)!}{v!}\right)^{\frac{1}{2v}}& \text{if $0\leq h\leq \frac{88v\log 2}{3\log T}$},\\
            \displaystyle 2C(1/88,v)^{\frac{1}{2v}}+\frac{\sqrt{\left|\log(\frac{h\log T}{88v/3})\right|+13.881}}{\sqrt{2}\pi}\left(\frac{(2v)!}{v!}\right)^{\frac{1}{2v}}& \text{if $\frac{88v\log 2}{3\log T}<h\leq 1$}.
        \end{cases}
    \end{equation}
    Here $C(\eps,v)$ was defined in Lemma~\ref{lemma S(t) dirich poly}. 
\end{proposition}

\begin{proof}
    By Minkowski's inequality, for any $\eps>0$,
    \begin{align*}
        &\left(\frac{1}{T}\int_{T}^{2T}\left|S(t+h)-S(t)\right|^{2v} \md t\right)^{\frac{1}{2v}}\\
        &\hspace{1cm}\leq \left(\frac{1}{T}\int_T^{2T} \left|S(t+h)+\frac{1}{\pi}\sum_{p\leq T^{3\eps/v}}\frac{\sin((t+h)\log p)}{\sqrt{p}}\right|^{2v}\md t\right)^{\frac{1}{2v}}\\
        &\hspace{2cm}+
        \left(\frac{1}{T}\int_T^{2T} \left|S(t)+\frac{1}{\pi}\sum_{p\leq T^{3\eps/v}}\frac{\sin(t\log p)}{\sqrt{p}}\right|^{2v}\md t\right)^{\frac{1}{2v}}\\
        &\hspace{2cm}+
        \left(\frac{1}{T}\int_T^{2T} \left|\frac{1}{\pi}\sum_{p\leq T^{3\eps/v}}\frac{\sin((t+h)\log p)-\sin(t\log p)}{\sqrt{p}}\right|^{2v}\md t\right)^{\frac{1}{2v}}.
    \end{align*}
    The first two integrals can be bounded using Lemma~\ref{lemma S(t) dirich poly}. By \eqref{int f(t)^{2k} 2}, the third integral does not exceed
    \begin{align*}
        \leq& \frac{T}{\pi^{2v}}\left[\frac{(2v)!}{4^v v!} \left(\sum_{p\leq T^{3\eps/v}}\frac{|1-p^{-ih}|^2}{p}\right)^v+B_1(v) \left(\sum_{p\leq T^{3\eps/v}} |1-p^{-ih}|^2\right)^v\right]\\
        =& \frac{T}{\pi^{2v}}\left[\frac{(2v)!}{2^v v!} \left(\sum_{p\leq T^{3\eps/v}}\frac{1-\cos(h\log p)}{p}\right)^v+B_1(v)4^v \left(\sum_{p\leq T^{3\eps/v}}(1-\cos(h\log p))\right)^v\right]\\
        =&\frac{T(2v)!}{\pi^{2v}2^vv!}\left(\sum_{p\leq T^{3\eps/v}}\frac{1-\cos(h\log p)}{p}\right)^v+O\left(\left(\frac{cv^2T^{3\eps/v}}{\eps\log T}\right)^v\right).
    \end{align*}
    Therefore, the quantity we seek to bound is at most 
    \begin{equation}\label{Prop 10 eqn}
        2C(\eps,v)^{\frac{1}{2v}}+\frac{1}{\sqrt{2}\pi}\left(\frac{(2v)!}{v!}\right)^{\frac{1}{2v}}\sqrt{\sum_{p\leq T^{3\eps/v}}\frac{1-\cos(h\log p)}{p}}+O\left(vT^{\frac{3\eps-1}{2v}}\right).
    \end{equation}
    When $v\leq \frac{\log T}{3\log\log T}$, the big-$O$ error term is $o(1)$ provided $\eps<1/9$. As for the second term, according to \cite[Lemma 6]{STT22}, if $X\geq 2$ and $0\leq h\leq \log 2/\log X$, then 
    \[
    \sum_{p\leq X}\frac{1-\cos(h\log p)}{p}\leq 2.02+\frac{3}{\log^2 X},
    \]
    and if $\log 2/\log X\leq h\leq 1$, then
    \[
    \left|\log(h\log X)-\sum_{p\leq X}\frac{1-\cos(h\log p)}{p}\right|\leq 13.88+\frac{3}{\log^2 X}.
    \]
    We obtain the proposition by taking $\eps=1/88<1/9$, $X=T^{3\eps/v}$ and $T$ sufficiently large. Note that the error term in \eqref{Prop 10 eqn} is absorbed by the slightly increased constants.
\end{proof}

The following crude bound on $\wt{C}(h,v)$ will also be useful later. 
\begin{lemma}\label{explicit bound C(h,v)}
    If $0\leq h\leq 1$, $h\log T\leq  e^{100v}$ and $13\leq v\leq \frac{\log T}{3\log\log T}$, then
    \[
    \wt{C}(h,v)<10^5v.
    \]

\end{lemma}
\begin{proof}
    From the definition of $C(\eps,v)$ we have
    \[
    C(1/88,v)^{\frac{1}{2v}}\leq 6^{-\frac{1}{2v}}\left(1+\sum_{n=1}^4 R_n(1/88,v)\right)^{\frac{1}{2v}}\cdot 12\cdot 88\cdot a_2\cdot a_4(1/88,v)\cdot v.
    \]
    One can check that the factor in front of $v$ on the right-hand side is decreasing in $v$, and hence evaluating at $v=13$ gives an upper bound $2C(1/88,v)^{\frac{1}{2v}}<9.96\cdot 10^4v$. 

    Next, we need an explicit version of Stirling's estimate. A very tight inequality due to Robbins \cite{Rob55}, which more than suffices for our purpose, states that 
    \begin{equation}\label{Robbins}
        \sqrt{2\pi n}\left(\frac{n}{e}\right)^n e^{\frac{1}{12n+1}}<n!<\sqrt{2\pi n}\left(\frac{n}{e}\right)^n e^{\frac{1}{12n}}, \qquad n\geq 1.
    \end{equation}
    This immediately implies that
    \begin{equation}\label{(2k)!/(k!) inequality}
        \frac{2}{\sqrt{e}}\sqrt{v}<\left(\frac{(2v)!}{v!}\right)^{\frac{1}{2v}}< \frac{2}{\sqrt{e}}2^\frac{1}{4v}\sqrt{v}.
    \end{equation}
    From this and the assumed bound $h\log T\leq e^{100v}$, we see that the second term in the definition of $\wt{C}(h,v)$ is at most $10v$ for both ranges of $h$. Putting these together yields the estimate stated in the lemma.
\end{proof}

\section{Moments of $R(t)$}\label{section: moments of R(t)}
Recall that $R(t)$ is defined as
\begin{equation*}\label{def R(t)}
    R(t)=2\sum_{\beta>1/2}\int_0^{\beta-1/2}\Im\left\{K_\tau(\gamma-t-h/2-i\sigma)-K_\tau(\gamma-t+h/2-i\sigma)\right\}\md \sigma.
\end{equation*}
Put $f(z)=\sin^2(z)/z^2$. It is not hard to verify that
\[
|f''(z)|=\left|\frac{(2z^2-3)\cos(2z)-4z\sin(2z)+3}{z^4}\right|<\frac{3e^{2|y|}}{1+x^2+y^2}
\]
for all $z=x+iy$. For large $|z|$ one may trivially apply the triangle inequality and use the fact that $|\cos(2z)|$ and $|\sin(2z)|$ are both $\leq \cosh(2y)\leq e^{2|y|}$. For small $|z|$ a direct numerical verification suffices. Since $K_\tau(z)=\frac{\tau}{2\pi}f(\frac{\tau z}{2})$, we have $K_\tau''(z)=\frac{\tau^3}{8\pi}f''(\frac{\tau z}{2})$. 
Hence for any $0\leq \sigma\leq \beta-1/2$,
\begin{align*}
    &\Im\left\{K_\tau(\gamma-t-h/2-i\sigma)-K_\tau(\gamma-t+h/2-i\sigma)\right\} \\ &\hspace{2cm}\leq \frac{3\tau^3}{8\pi}\int_{\gamma-t-h/2}^{\gamma-t+h/2}\int_0^\sigma \frac{e^{\tau y}}{1+(\tau x/2)^2+(\tau y/2)^2}\md y \md x\\
    &\hspace{2cm}\leq \frac{3\tau^3\sigma e^{\tau\sigma}}{8\pi} \int_{\gamma-t-h/2}^{\gamma-t+h/2} \frac{1}{1+(\tau x/2)^2}\md x\\
    &\hspace{2cm}=\frac{3\tau^2\sigma e^{\tau\sigma}}{4\pi}  \int_{\tau(\gamma-t-h/2)/2}^{\tau(\gamma-t+h/2)/2} \frac{1}{1+x^2}\md x\\
    &\hspace{2cm}\leq \frac{3\tau^3 h\sigma e^{\tau\sigma}}{8\pi} \frac{E(\tau h/4)}{1+(\tau(\gamma-t)/2)^2},
\end{align*}
where
\begin{equation}\label{def E(delta)}
    E(\delta):=\sup_{u\in \bb{R}}\frac{1+u^2}{2\delta}\int_{u-\delta}^{u+\delta}\frac{1}{1+x^2}\md x.
\end{equation}
This gives
\begin{align*}
    |R(t)|\leq& \frac{3\tau^3h}{4\pi}E(\tau h/4) \sum_{\beta>1/2}\int_0^{\beta-1/2}\frac{\sigma e^{\tau \sigma}}{1+(\tau(\gamma-t)/2)^2} \md\sigma\\
    \leq& \frac{3\tau^3 h}{4\pi}E(\tau h/4)\sum_{\beta>1/2}\frac{4\tau^{-2}}{2}\frac{(\beta-1/2)^2 e^{\tau(\beta-1/2)}}{4\tau^{-2}+(\gamma-t)^2}\\
    =& \frac{3\tau h}{2\pi} E(\tau h/4)\sum_{\beta>1/2}\frac{1}{[4\tau^{-2}+(\gamma-t)^2]^{(2v-1)/(2v)}}\frac{(\beta-1/2)^2 e^{\tau(\beta-1/2)}}{[4\tau^{-2}+(\gamma-t)^2]^{1/(2v)}}.
\end{align*}
Define
\[
\wt{R}(t):=\sum_{\beta>1/2}\frac{1}{4\tau^{-2}+(\gamma-t)^2}.
\]
Let us assume that $v\tau\leq \frac{1}{10}\log T$. By H\"{o}lder's inequality applied with $p=2v/(2v-1)$ and $q=2v$, we obtain
\begin{align*}
    |R(t)|^v\leq& \left(\frac{3}{2\pi}\tau h E(\tau h/4)\right)^v |\wt{R}(t)|^{v-1/2} \sqrt{\sum_{\beta>1/2}\frac{(\beta-1/2)^{4v}e^{2v\tau(\beta-1/2)}}{4\tau^{-2}+(\gamma-t)^2}}\\
    \leq& \left(\frac{3}{2\pi}\tau h E(\tau h/4)\right)^v |\wt{R}(t)|^{v-1/2} \sqrt{\sum_{\substack{\beta>1/2\\ |\gamma-t|\leq \sqrt{T}}}\frac{(\beta-1/2)^{4v}e^{2v\tau(\beta-1/2)}}{4\tau^{-2}+(\gamma-t)^2}+O(T^{-1/3})}.
\end{align*}
Here we have bounded the complementary sum by
\[
\ll e^{v\tau}\sum_{|\gamma-t|>\sqrt{T}}(\gamma-t)^{-2}\ll e^{v\tau}T^{-1/2}\log T\ll T^{-1/3}.
\]
It now follows from Cauchy--Schwarz that
\begin{align}\label{moment of R}
    \int_T^{2T}|R(t)|^v\md t\leq& \left(\frac{3}{2\pi}\tau h E(\tau h/4)\right)^v\sqrt{\int_T^{2T}|\wt{R}(t)|^{2v-1}\md t}\notag\\
    &\hspace{1cm}\times \sqrt{\sum_{\substack{\beta>1/2\\ \gamma\in [T-\sqrt{T},2T+\sqrt{T}]}}(\beta-1/2)^{4v}e^{2v\tau(\beta-1/2)}\cdot \int_T^{2T}\frac{1}{4\tau^{-2}+(\gamma-t)^2}\md t+O(T^{2/3})}\notag\\
    \leq& \left(\frac{3}{2\pi}\tau h E(\tau h/4)\right)^v \left(T\int_T^{2T}|\wt{R}(t)|^{4v-2}\md t\right)^{1/4}\sqrt{\frac{\pi \tau}{2}} \notag\\
    &\hspace{1cm}\times \sqrt{\sum_{\substack{\beta>1/2\\ \gamma\in [T-\sqrt{T},2T+\sqrt{T}]}}(\beta-1/2)^{4v}e^{2v\tau(\beta-1/2)}+O(T^{2/3})}.
\end{align}
It remains to estimate the integral and the zero sum.
\begin{lemma}\label{lemma moment of R}
    For $v\in \bb{N}$, $v\leq \frac{\log T}{6\log\log T}$,
    \begin{align*}
        \frac{1}{T}\int_{T}^{2T}|\wt{R}(t)|^{4v-2}\md t \leq \left(\frac{0.911\tau\log T}{2\pi}+0.911\tau^2 \wt{C}(\tau^{-1},2v-1)\right)^{4v-2}
    \end{align*}
    where $\wt{C}$ was defined in \eqref{def tilde C}.
\end{lemma}
\begin{proof}
    By the symmetry of zeros,
    \[
    \wt{R}(t)<\frac{1}{2}\sum_\gamma \frac{1}{4\tau^{-2}+(\gamma-t)^2}.
    \]
    A standard argument shows that
    \[
    \sum_{|\gamma-t|\geq \log T}\frac{1}{(\gamma-t)^2}=O(1).
    \]
    Dividing the complementary sum into subintervals, we obtain
    \begin{align*}
        \wt{R}(t)<&\frac{1}{2}\sum_{n=0}^{\lfloor \tau \log T  \rfloor}\sum_{\frac{n}{\tau}\leq |\gamma-t|<\frac{n+1}{\tau}}\frac{1}{4\tau^{-2}+(\gamma-t)^2}+O(1)\\
        \leq& \frac{\tau^2}{2}\sum_{n=0}^{\lfloor \tau \log T  \rfloor}\frac{1}{n^2+4}\left[N\left(t+\frac{n+1}{\tau}\right)-N\left(t+\frac{n}{\tau}\right)+N\left(t-\frac{n}{\tau}\right)-N\left(t-\frac{n+1}{\tau}\right)\right]+O(1)\\
        =&\frac{\tau^2}{2}\sum_{n=0}^{\lfloor \tau\log T\rfloor}\frac{1}{n^2+4} \bigg[\frac{\log T}{\pi\tau}+S\left(t+\frac{n+1}{\tau}\right)-S\left(t+\frac{n}{\tau}\right)\\
        &\hspace{6cm}+S\left(t-\frac{n}{\tau}\right)-S\left(t-\frac{n+1}{\tau}\right)\bigg]+O(\tau).
    \end{align*}
    Using the identity
    \[
    \sum_{n=0}^\infty \frac{1}{n^2+4}=\frac{\pi}{4}\mr{coth}(2\pi)+\frac{1}{8}=0.9104\ldots<0.911
    \]
    and H\"{o}lder's inequality with $p=\frac{2v}{2v-1}$ and $q=2v$, we see that for all large $T$, 
    \begin{align*}
        |\wt{R}(t)|\leq &\frac{0.911\tau \log T}{2\pi}+\frac{\tau^2}{2}\sum_{n=0}^{\lfloor \tau \log T \rfloor}\frac{1}{(n^2+4)^{\frac{2v-1}{2v}}} \frac{1}{(n^2+4)^{\frac{1}{2v}}}\bigg[S\left(t+\frac{n+1}{\tau}\right)-S\left(t+\frac{n}{\tau}\right)\\
        &\hspace{6cm}+S\left(t-\frac{n}{\tau}\right)-S\left(t-\frac{n+1}{\tau}\right)\bigg]\\
        \leq& \frac{0.911\tau \log T}{2\pi}+\frac{\tau^2}{2}0.911^{\frac{2v-1}{2v}}\Bigg(\sum_{n=0}^{\lfloor \tau\log T\rfloor}\frac{1}{n^2+4}\bigg[S\left(t+\frac{n+1}{\tau}\right)-S\left(t+\frac{n}{\tau}\right)\\
        &\hspace{6cm}+S\left(t-\frac{n}{\tau}\right)-S\left(t-\frac{n+1}{\tau}\right)\bigg]^{2v}\Bigg)^{\frac{1}{2v}}.
    \end{align*}
    Applying Minkowski's inequality and the inequality $(a+b)^{2v}\leq 2^{2v-1}(a^{2v}+b^{2v})$ yields
    \begin{multline*}
        \|\wt{R}(t)\|_{2v}\leq T^{\frac{1}{2v}}\frac{0.911\tau\log T}{2\pi}+\tau^2 \frac{0.911}{2^{\frac{1}{2v}}}\max_{0\leq n\leq \lfloor \tau\log T\rfloor}\Bigg(\int_T^{2T} \bigg[S\left(t+\frac{n+1}{\tau}\right)-S\left(t+\frac{n}{\tau}\right)\bigg]^{2v}\\
        +\bigg[S\left(t-\frac{n}{\tau}\right)-S\left(t-\frac{n+1}{\tau}\right)\bigg]^{2v}\md t\Bigg)^{\frac{1}{2v}}.
    \end{multline*}
    The claimed estimate follows upon replacing $v$ by $2v-1$ and then invoking Proposition~\ref{prop: moment S(t+h)-S(t)}.
\end{proof}

\begin{lemma}\label{lemma sum over zeros}
    Let $L$ be the constant from \eqref{zero density}. If $2v\tau\leq \frac{1}{5}\log T$, then
    \[
    \sum_{\substack{\beta>1/2\\ \gamma \in [T-\sqrt{T},2T+\sqrt{T}]}}(\beta-1/2)^{4v}e^{2v\tau(\beta-1/2)}\leq \frac{LvT\log T}{(\frac{1}{4}\log T-2v\tau)^{4v}}\left(\frac{2(4v)!\tau}{\frac{1}{4}\log T-2v\tau}+4(4v-1)!+o(1)\right).
    \]
\end{lemma}
\begin{proof}
    The sum in question can be written as
    \begin{multline*}
        -\int_{1/2}^1 (\sigma-1/2)^{4v} e^{2v\tau(\sigma-1/2)}\md N(\sigma;T-\sqrt{T},2T+\sqrt{T})\\
        \leq v\int_{1/2}^1 N(\sigma;T-\sqrt{T},2T+\sqrt{T}) (\sigma-1/2)^{4v-1}(2\tau \sigma-\tau+4)e^{2v\tau(\sigma-1/2)} \md \sigma
    \end{multline*}
    via integration by parts. By \eqref{zero density} and Lemma~\ref{lemma zero density}, this does not exceed
    \begin{align*}
        \leq &LvT\log T \int_{1/2}^{\sigma_0} (\sigma-1/2)^{4v-1}(2\tau \sigma-\tau+4)e^{(2v\tau-\frac{1}{4}\log T)(\sigma-1/2)} \md \sigma\\
        &\hspace{0.5cm}+O\left(v\sqrt{T}\log T\int_{1/2}^{\sigma_0} (\sigma-1/2)^{4v-1}(2\tau \sigma-\tau+4)e^{2v\tau(\sigma-1/2)} \md \sigma\right)\\
        &\hspace{0.5cm}+O\left(vT\log T\int_{\sigma_0}^{1} (\sigma-1/2)^{4v-1}(2\tau \sigma-\tau+4)e^{(2v\tau-\frac{1}{4}\log T)(\sigma-1/2)} \md \sigma\right)\\
        \leq& \frac{LvT\log T}{(\frac{1}{4}\log T-2v\tau)^{4v}}\left(\frac{2\tau\Gamma(4v+1)}{\frac{1}{4}\log T-2v\tau}+4\Gamma(4v)\right)\\
        &\hspace{0.5cm}+O\left(T^{2/3}+\frac{vT\log T}{(\frac{1}{4}\log T-2v\tau)^{4v}}\left(\frac{2\tau\Gamma(4v+1,x_0)}{\frac{1}{4}\log T-2v\tau}+4\Gamma(4v,x_0)\right)\right),
    \end{align*}
    where 
    \[
    x_0=(\sigma_0-1/2)\left(\frac{1}{4}\log T-2v\tau\right),
    \]
    $\Gamma(a)=\int_0^\infty t^{a-1}e^{-t}\md t$ and $\Gamma(a,x)=\int_x^\infty t^{a-1}e^{-t}\md t$ are the Gamma function and the incomplete Gamma function, respectively. Note that for fixed $a$, $\Gamma(a,x)$ decays rapidly as $x\to \infty$. The big-$O$ error term is therefore absorbed into the $o(1)$ term in the stated estimate.
\end{proof}

Inserting the bounds from Lemma~\ref{lemma moment of R} and ~\ref{lemma sum over zeros} into \eqref{moment of R}, we obtain for $v\tau\leq \frac{1}{10}\log T$ and $\tau\geq \log\log T$ (which together implies the requisite condition for Lemma~\ref{lemma moment of R}),
\begin{align}\label{M3}
    \frac{1}{T}\int_T^{2T}|R(t)|^v\md t\leq& \left(\frac{3}{2\pi}\tau h E(\tau h/4)\right)^v \left(\frac{0.911\tau\log T}{2\pi}+0.911\tau^2 \wt{C}(\tau^{-1},2v-1)\right)^{v-1/2}\notag\\
    &\hspace{1cm}\times \sqrt{\frac{\pi\tau}{2}}\frac{\sqrt{Lv\log T}}{(\frac{1}{4}\log T-2v\tau)^{2v}}\sqrt{\frac{2(4v)!\tau}{\frac{1}{4}\log T-2v\tau}+4(4v-1)!}.
\end{align}

Before proceeding to the next section, we give a quick estimate for the function $E(\delta)$ defined in \eqref{def E(delta)}. This will help with numerical calculations later.
\begin{lemma}\label{lemma E}
    Let $\delta>0$.
    \begin{enumerate}[label=(\alph*)]
        \item 
        \[
        E(\delta)\leq 1+\frac{\delta^2}{6}+\frac{(\delta^2+4)^{3/2}-8}{6\delta};
        \]
        
        \item 
        \[
        E(\delta)=\frac{1+{u_*}^2}{2\delta}\left(\arctan(u_*+\delta)-\arctan(u_*-\delta)\right),
        \]
        where $u_*$ is the unique positive root of the equation
        \[
        \arctan(u+\delta)-\arctan(u-\delta)-\frac{2\delta(1+u^2)}{(1+(u+\delta)^2)(1+(u-\delta)^2)}=0.
        \]
    \end{enumerate}
\end{lemma}
\begin{proof}
    By definition,
    \begin{align*}
        E(\delta)=&\frac{1}{2\delta}\sup_{u\in \bb{R}}\int_{u-\delta}^{u+\delta}\frac{1+u^2}{1+x^2}\md x\leq \frac{1}{2\delta}\int_{-\delta}^{\delta}\sup_{u\in \bb{R}}\frac{1+u^2}{1+(x+u)^2}\md x=\frac{1}{2\delta}\int_{-\delta}^{\delta}\left(\frac{|x|+\sqrt{x^2+4}}{2}\right)^2\md x.
    \end{align*}
    Evaluating the last integral yields part (a) of the lemma. For part (b), we write $E(\delta)=\sup_u h(u)$ where 
    \[
    h(u)=\frac{1+u^2}{2\delta}(\arctan(u+\delta)-\arctan(u-\delta)).
    \]
    Since $h(u)$ is even, it suffices to restrict to $u\geq 0$. We find that 
    \[
    \frac{\delta h'(u)}{u}=\arctan(u+\delta)-\arctan(u-\delta)-\frac{2\delta(1+u^2)}{(1+(u+\delta)^2)(1+(u-\delta)^2)}
    \]
    and that
    \[
    \frac{\md}{\md u}\left(\frac{\delta h'(u)}{u}\right)=-\frac{8\delta^3 u(\delta^2+3-u^2)}{(1+(u+\delta)^2)(1+(u-\delta)^2)},
    \]
    which has a unique positive root $\sqrt{\delta^2+3}$. Moreover, $h'(0)>0$ and $h'(u)<0$ for large $u$. It follows that $h'(u)$ has a unique positive root $u_*$ with $h'(u)>0$ for $0\leq u<u_*$ and $h'(u)<0$ for $u>u_*$, and hence the supremum of $h(u)$ is attained at $u_*$.
\end{proof}

\section{Completing the Proofs of Theorems~\ref{theorem measure S(t+h)-S(t)} and ~\ref{theorem 2}}\label{section: completing the proof}

We are now ready to assemble the various moment estimates and complete the proofs of our main theorems by following the argument outlined in \S\ref{section: reduction}.

\subsection{Proof of Theorem~\ref{theorem measure S(t+h)-S(t)}}
Let $h\in [\frac{1}{\log T},\frac{\kappa}{\log\log T}]$ for a possibly small constant $\kappa>0$. We start by recording the expressions of $M_i$ as determined in Lemma~\ref{key lemma}. Set $\tau=\eta h^{-1}$ for some parameter $0<\eta<2\pi$ such that 
\begin{equation}\label{parameters restrictions}
    h\log k\leq \eta/2, \quad \tau\leq \frac{\log T}{k+1/2}-\log\log T, \quad \eta \geq h\log\log T, \quad (8k+4)\eta\leq \frac{1}{5}h\log T.
\end{equation}
The first two conditions are imposed for $W(t)$ (see \S\ref{section: moments of W(t)}), and the last two for $R(t)$ with $v=4k+2$ (see \S\ref{section: moments of R(t)}), which may be dropped if we assume RH. Note that the second condition in fact follows from the fourth. In view of \eqref{M1}, \eqref{M2} and \eqref{M3}, we have
\begin{align}\label{M_i expression}
        M_0=&2^{-\frac{1}{2k+1}}M_1-M_3,\notag\\
        M_1=&\frac{1}{\pi}
        \left(\frac{(2k)!}{k!}\right)^{\frac{1}{2k}}\sqrt{\int_{h\log k}^{\eta}\left(1-\frac{u}{\eta}\right)^2\frac{\sin^2(u/2)}{u}\md u},\notag\\
        M_2=&\frac{1}{\pi}
        \left(\frac{(4k+2)!}{(2k+1)!}\right)^{\frac{1}{4k+2}}\sqrt{\int_{0}^{\eta}\left(1-\frac{u}{\eta}\right)^2\frac{\sin^2(u/2)}{u}\md u},\notag\\
        M_3=&\frac{3}{2\pi}\eta E(\eta/4)\left(\frac{0.911\eta\cdot h\log T}{2\pi}+0.911\eta^2 \wt{C}(h\eta^{-1},4k+1)\right)^{1-\frac{1}{4k+2}}\\
        &\hspace{0.5cm}\times 
        \frac{(\pi(k+1/2)L\eta\cdot h\log T)^{\frac{1}{4k+2}}}{(\frac{1}{4}h\log T-(4k+2)\eta)^2}\left(\frac{2\eta(8k+4)!}{\frac{1}{4}h\log T-(4k+2)\eta}+4(8k+3)!\right)^{\frac{1}{4k+2}},\notag\\
        M_4=&\frac{3}{2\pi}\eta E(\eta/4) \left(\frac{0.911\eta\cdot h\log T}{2\pi}+0.911\eta^2 \wt{C}(h\eta^{-1},8k+3)\right)^{1-\frac{1}{8k+4}} \notag\\
        &\hspace{0.5cm}\times 
        \frac{(\pi(2k+1)L\eta\cdot h\log T)^{\frac{1}{8k+4}}}{(\frac{1}{4}h\log T-(8k+4)\eta)^2}\left(\frac{2\eta(16k+8)!}{\frac{1}{4}h\log T-(8k+4)\eta}+4(16k+7)!\right)^{\frac{1}{8k+4}}\notag,
\end{align}
where $E$ was defined in \eqref{def E(delta)}, $\wt{C}$ was defined in \eqref{def tilde C} and $L=3.276$ is admissible according to Lemma~\ref{lemma zero density}. Since $\eta=O(1)$, \eqref{(2k)!/(k!) inequality} yields 
\[
M_1\asymp\eta \sqrt{k} \qquad \text{and} \qquad M_2\asymp \eta\sqrt{2k}.
\]
By Lemma~\ref{lemma E}, $E(\eta/4)=1+O(\eta^2)$. Also note that if $v\eta\ll h\log T$, then 
\[
\eta^2\wt{C}(h\eta^{-1},v)\ll\eta^2(v+\sqrt{v\log(3+h\eta^{-1}\log T)})\ll \eta h\log T,
\]
and so
\[
M_3\asymp \eta^2\frac{k^2}{h\log T} \qquad \text{and} \qquad  M_4\asymp \eta^2 \frac{(2k)^2}{h\log T}.
\]
Thus, when $\eta$ is small enough (but fixed) and $k\asymp (h\log T)^{2/3}$, we have $M_0\asymp \sqrt{k}$. In particular, the largest $M_0$ can be is $\asymp (h\log T)^{1/3}$. Our range of $k$ satisfies the fourth condition in \eqref{parameters restrictions} and does not conflict with the first condition $h\log k\leq \eta/2$ provided $h\leq \frac{\kappa}{\log\log T}$ for some constant $\kappa>0$, which can be adjusted so that the chosen $\eta$ also satisfies the third requirement $\eta\geq h\log\log T$. Furthermore, by Proposition~\ref{prop: moment S(t+h)-S(t)}, the quantity $A_0$ defined in \eqref{square S(t+h)-S(t)} satisfies
\begin{equation}\label{bound A_0}
    A_0\leq \wt{C}(h,1)^2+O\left(\frac{\log^3 T}{T}\right)=\frac{1}{\pi^2}\log(h\log T+2)+O(1).
\end{equation}
Here we used $S(t)=O(\log t)$ again to control the integral over $[T-\log T, T]$ and $[2T+\log T]$. 

Let $C>0$ be a constant such that $\sup_{k,\eta} M_0\geq C(h\log T)^{1/3}$, and let $0<V\leq C'(h\log T)^{1/3}$ where $C'<C$ is fixed. Returning to \eqref{measure I_V}, we find that the measure of the subset of $t\in [T,2T]$ for which $S(t+h)-S(t)\geq V$ is at least
\[
\gg \frac{TM_0^2}{A_0}\left(\frac{M_1}{M_2}\right)^{8k+4}\gg \frac{Tk}{A_016^k}\gg \frac{Te^{-DV^2}}{\log(h\log T+2)}
\]
upon choosing $k=\lfloor D'V^2\rfloor$, where $D,D'>0$ are appropriate constants.

Now assume RH so that $R(t)=M_3=M_4=0$. In this case, $k$ can be taken as large as $\asymp h\log T$, and so $\sup_{k,\tau} M_0\gg (h\log T)^{1/2}$, which is the conditional upper limit for the admissible range of $V$. 

Since the same argument works for $-(S(t+h)-S(t))$, the proof of Theorem~\ref{theorem measure S(t+h)-S(t)} is complete.

\subsection{Proof of Theorem~\ref{theorem 2}}
It remains to explicitly calculate the constants presented in Theorem~\ref{theorem 2}. All numerical computations are carried out on Mathematica. In this setting, $h$ is roughly $\frac{2\pi r}{\log T}$, and from the previous proof we see that $\sup_{k,\tau}M_0\asymp r^{1-\alpha}$ by choosing $k\asymp r^{2(1-\alpha)}$. For any given $r\geq 1$, all of $k,M_i,A_0$ are bounded as $T\to \infty$. Hence, recall from \S\ref{section: r-gaps} that 
\[
\wt{\lambda}_r\geq 1+\frac{\Theta(r)}{r^\alpha}\qquad \text{and} \qquad \wt{\mu}_r\leq 1-\frac{\vartheta(r)}{r^\alpha}
\]
where
\[
\Theta(r):=\sup_{k,\eta} \left\{\theta:\theta<\frac{M_0(k,h,\eta)}{r^{1-\alpha}}, h=\frac{2\pi r(1+\theta r^{-\alpha})}{\log T}\right\}.
\]
and
\[
\vartheta(r):=\sup_{k,\eta} \left\{\theta:\theta<\frac{M_0(k,h,\eta)}{r^{1-\alpha}}, h=\frac{2\pi r(1-\theta r^{-\alpha})}{\log T}\right\}.
\]

\subsubsection{The unconditional case} We first deal with the case $r=1$. In order to maximize $M_0$, we take $k=1$ and $\eta=7.1\cdot 10^{-5}$ in \eqref{M_i expression}. Note that our choice of $k$ and $\eta$ satisfies \eqref{parameters restrictions}. As a result,
\[
\Theta(1)>\vartheta(1)>2.36\cdot 10^{-6}.
\]
Furthermore, we can use \eqref{proportion f1} and \eqref{proportion f2} to compute lower bounds for the proportions of small and large gaps. Consider small gaps for example. Note that $A_1(c,\theta)r^{1/3}$ remains small (certainly $<1$) for all $c$ in the admissible range, so we take $c$ to be arbitrarily close to its smallest admissible value $c_2^-=\theta/[2\pi(1-\theta r^{-2/3})]^{1/3}$. With $r=1$, we arrive at the simplified lower bound
\[
D^-(1-\theta,1)\geq \liminf_{T\to \infty} \sup_{\substack{k,\eta\\M_0>\theta}}\frac{2^4\big(M_0-\theta \big)^6}{3^6 \wt{C}(h,1)^2 M_0^4}\\
\left(\frac{\frac{1}{2}M_1^{2k+1}-M_3^{2k+1}}{M_2^{2k+1}+M_4^{2k+1}}\right)^4
\frac{1}{1-\theta}
\]
where we used the first inequality in \eqref{bound A_0}. For large gaps, we obtain
\[
D^+(1+\theta,1)\leq 1-\liminf_{T\to \infty}\sup_{\substack{k,\eta\\M_0>\theta}}\frac{2^8\big(M_0-\theta \big)^{12}}{4\pi e^{99.8} \cdot 3^{12} \wt{C}(h,1)^4 M_0^8}\\
\left(\frac{\frac{1}{2}M_1^{2k+1}-M_3^{2k+1}}{M_2^{2k+1}+M_4^{2k+1}}\right)^8.
\]
For instance, for $\theta=10^{-6}$, the choice $k=1$ and $\eta=5.4\cdot 10^{-5}$ yields
\[
D^-(1-10^{-6},1)>4.5\cdot 10^{-30}, \qquad D^+(1+10^{-6},1)<1-7.5\cdot 10^{-104}.
\]
It is worth pointing out that decreasing $\theta$ does not improve the corresponding proportions by much. In fact, even as we take $\theta\to 0^+$, the best lower bounds this method produces for $D^-(1-\theta,1)$ and $1-D^+(1+\theta,1)$ are only about $2.2\cdot 10^{-28}$ and $1.7\cdot 10^{-100}$, respectively (with $k=1$ and $\eta=4.5\cdot 10^{-5}$).

Next, we verify that $\Theta(r)$ and $\vartheta(r)$ are $>2.36\cdot 10^{-6}$ for all $r\geq 1$. Take $k=\lfloor r^{2/3}\rfloor$ and $\eta=7.1\cdot 10^{-5}$. It is straightforward to check that this choice works for $1\leq r\leq 20$, so we suppose that $r>20$, for which $k\geq 7$. Then, by \eqref{(2k)!/(k!) inequality},
\[
2^{-\frac{1}{2k+1}}M_1>2^{-\frac{1}{2k+1}}\frac{2\sqrt{k}}{\pi\sqrt{e}}\sqrt{\int_0^\eta \left(1-\frac{u}{\eta}\right)^2\frac{\sin^2(u/2)}{u}\md u}>3\cdot 10^{-6}\cdot r^{1/3},
\]
By Lemma~\ref{explicit bound C(h,v)} (which is applicable since $v=4k+1>13)$ and \eqref{Robbins},
\begin{align*}
    M_3\leq &\frac{3}{2\pi}\eta E(\eta/4)\left(0.911\eta r(1+\theta r^{-2/3})+0.911\eta^2 10^5 (4k+1)\right)^{1-\frac{1}{4k+2}} \\
    &\hspace{0.5cm}\times 
    \frac{(\pi(k+1/2)\cdot 3.276\eta\cdot 2\pi r(1+\theta r^{-2/3}))^{\frac{1}{4k+2}}}{(\pi r(1+\theta r^{-2/3})/2-(4k+2)\eta)^2}\left(\frac{\eta(16k+8)}{\pi r(1+\theta r^{-2/3})/2-(4k+2)\eta}+4\right)^{\frac{1}{4k+2}}\\
    &\hspace{0.5cm}\times \left(\left(\frac{8k+3}{e}\right)^{8k+3}\sqrt{2\pi(8k+3)}e^{\frac{1}{12(8k+3)}}\right)^{\frac{1}{4k+2}}.
\end{align*}
With the bound on $E(\eta/4)$ from Lemma~\ref{lemma E}(a) and a little calculation, we readily deduce that
\[
M_3<3.5\cdot 10^{-5}\cdot (10^{-3}r)^{1-\frac{1}{4k+2}}\frac{r^{\frac{1}{4k+2}}}{(1.57r)^2}\cdot 2 \left(\frac{8r^{2/3}+3}{e}\right)^2<3\cdot 10^{-7}\cdot r^{1/3}.
\]
It follows that $\Theta(r)>2.7\cdot 10^{-6}$ for all $r> 20$, and similarly $\vartheta(r)$ satisfies the same bound. This proves that $\Theta$ and $\vartheta$ can both be taken as $2.36\cdot 10^{-6}$.

As $r\to \infty$, for $k\asymp r^{2/3}$ we have
\[
2^{-\frac{1}{2k+1}}M_1\sim \frac{2\sqrt{k}}{\pi\sqrt{e}}\sqrt{\int_0^\eta \left(1-\frac{u}{\eta}\right)^2\frac{\sin^2(u/2)}{u}\md u}
\]
and
\[
M_3\sim \frac{3}{2\pi}E(\eta/4)\frac{0.911\eta^2 r}{(\pi r/2)^2}\left(\frac{8k}{e}\right)^2<1.527 E(\eta/4)\eta^2\frac{k^2}{r}.
\]
Choosing $k=\delta r^{2/3}$ and optimizing over $\delta>0$ in terms of $\eta$ gives
\[
\liminf_{r\to \infty}\frac{M_0}{r^{1/3}}>\frac{3\left(\frac{2}{\pi\sqrt{e}}\right)^{4/3}}{4(4\cdot 1.527 E(\eta/4)\eta^2)^{1/3}}\left(\int_0^\eta \left(1-\frac{u}{\eta}\right)^2\frac{\sin^2(u/2)}{u}\md u\right)^{2/3}.
\]
Finally, we take $\eta=5.4<2\pi$, for which a numerical computation using Lemma~\ref{lemma E}(b) shows that $E(\eta/4)=1.3216\ldots$, to conclude that 
\[
\liminf_{r\to \infty}\Theta(r)=\liminf_{r\to \infty}\vartheta(r)>0.01834.
\]

\subsubsection{The conditional case} Now assume RH. Taking $k=r$ and $\tau$ to be its largest possible value according to \eqref{parameters restrictions},
we have
\[
M_0=2^{-\frac{1}{2r+1}}\frac{1}{\pi}
\left(\frac{(2r)!}{r!}\right)^{\frac{1}{2r}}\sqrt{\int_{0}^{\eta_r}\left(1-\frac{u}{\eta_r}\right)^2\frac{\sin^2(u/2)}{u}\md u}
\]
where
\[
\eta_r=\min\left\{2\pi,\frac{2\pi r(1\pm\theta /\sqrt{r})}{r+1/2}\right\}
\]
according as $h=\frac{2\pi r(1\pm \theta/\sqrt{r})}{\log T}$. This gives
\[
\Theta(1)=0.2160\ldots \qquad \text{and} \qquad \vartheta(1)=0.1638\ldots
\]
(which is of course superseded by existing conditional results for $r=1$ in the literature; see the introduction section). Inoue calculated these constants using the same method in \cite{Inoue24}, but obtained a pair of slightly erroneous values $\Theta(1)=0.3224\ldots$ and $\vartheta(1)=0.2279\ldots$ (see his Table 2 on page 2272) due to a missing variable $u$ in the integral that appears in the above expression for $M_0$ (see his display equations (4.2) and (4.3) on page 2275). One can also check that $\Theta(r)$ and $\vartheta(r)$ are increasing in $r$, so that $\Theta=0.2160$ and $\vartheta=0.1638$ are admissible under RH. As $r\to \infty$, $\Theta(r)$ and $\vartheta(r)$ both approach 
\[
\lim_{r\to \infty}\frac{M_0}{\sqrt{r}}=\frac{2}{\pi\sqrt{e}}\sqrt{\int_0^{2\pi}\left(1-\frac{u}{2\pi}\right)^2\frac{\sin^2(u/2)}{u}\md u}=0.2643\ldots.
\]
As for the proportions, we again use \eqref{proportion f1} and \eqref{proportion f2} to derive, as an example,
\[
D^-(1-0.15,1)>1.9\cdot 10^{-24}, \qquad D^+(1+0.15,1)<1-2.6\cdot 10^{-43}.
\]
The proof of Theorem~\ref{theorem 2} is now complete. 

We end with a few remarks.

\begin{remark}\label{possible improvement}
    Under RH, Fujii \cite[Theorem 2 and Corollary]{Fuj90} gave a sharper bound on $A_0$ defined in \eqref{square S(t+h)-S(t)}, which we bounded unconditionally in Proposition~\ref{prop: moment S(t+h)-S(t)} as the special case $v=1$. This would improve the conditional proportion results. Unconditionally, the proportions can be improved by refining Korol\"{e}v's inequality \eqref{Korolev bound} using Proposition~\ref{prop: moment S(t+h)-S(t)}. We choose not to include these refinements here. 
\end{remark}

\begin{remark}
    We briefly remark on the reason why the constants $\Theta$ and $\vartheta$ in Theorem~\ref{theorem 2} are much larger than those in \eqref{sim-trud-turn constants}. The proof of \eqref{sim-trud-turn constants} in \cite{STT22} critically relies on a lower (rather than upper) bound for the first moment $\int_T^{2T}|S(t+h)-S(t)|\md t$, which is derived from a lower bound on the second moment and an upper bound on the fourth moment via a standard application of H\"{o}lder's inequality. Tsang \cite[Theorem 3]{Tsa86} established the asymptotic formula  
    \begin{multline*}
        \int_T^{2T}|S(t+h)-S(t)|^{2k}\md t=\frac{(2k)!}{2^k\pi^{2k}k!}T\log^k(h\log T+2)\\
        +O\left(T(ck)^k\left(k^k+\log^{k-1/2}(h\log T+2)\right)\right)
    \end{multline*}
    for $0<h<1$, where the error term is made explicit in \cite{STT22} for $k=1,2$. Unfortunately, observe that the error term is of the same order of magnitude as the main term when $h\log T=O(1)$, so that one needs to take $h$ much larger than the average spacing
    in order to get a non-trivial lower bound for the left-hand side. However, the size of their final constant is roughly the ratio of this lower bound to $h\log T$, which turns out to be double exponentially small. While we have also utilized the same moments in our argument, what we need is only an upper bound (see \S\ref{section: moments of S(t+h)-S(t)} and \S\ref{section: moments of R(t)}), and therefore we do not care whether the error term is numerically dominated by the main term.
    
\end{remark}

\section{Acknowledgement}
I would like to thank Professor Ghaith Hiary and Professor Timothy Trudgian for their comments and encouragement. I am also grateful to Lukas Kornmann for pointing out a few mistakes in an earlier draft and for the subsequent corrections and discussions.

\printbibliography
\end{document}